\documentclass[a4paper, 11pt, reqno]{amsart}
\usepackage{amsthm,amssymb,amsmath,enumerate,graphicx,psfrag,enumitem,mathrsfs}

\usepackage[margin=2.5cm]{geometry}
\usepackage{hyperref}
\usepackage{array}
\hypersetup{colorlinks=true,
   citecolor=blue,
   filecolor=blue,
   linkcolor=blue,
   urlcolor=blue
}

\usepackage{tikz}
\usepackage[textsize=footnotesize]{todonotes}

\usepackage{algorithm}
\usepackage{tikz}

\usetikzlibrary{backgrounds}

\newtheorem{definition}{Definition}[section]
\newtheorem{claim}{Claim}
\newtheorem{proposition}[definition]{Proposition}
\newtheorem{theorem}[definition]{Theorem}
\newtheorem{corollary}[definition]{Corollary}
\newtheorem{lemma}[definition]{Lemma}

\newtheorem{question}[definition]{Question}

\numberwithin{equation}{section}

\newcommand{\comment}[1]{}

\newcommand{\cS}{\mathcal{S}}
\newcommand{\cF}{\mathcal{F}}
\newcommand{\cT}{\mathcal{T}}

\newcommand{\cH}{\mathcal{H}}

\newcommand{\cL}{\mathcal{L}}

\newcommand{\bG}{\mathbf{G}}

\newcommand{\den}{{\rm den}}

\renewcommand{\epsilon}{\varepsilon}

\newcommand{\COMMENT}[1]{}
\renewcommand{\COMMENT}{\footnote} 

\floatstyle{plain}
\restylefloat{figure}

\title{Tree decompositions of graphs without large bipartite holes}
\author{Jaehoon Kim, Younjin Kim and Hong Liu}
\thanks{J.K.\ was supported by ERC grant~306349; Y.K.\ was supported by Basic Science Research Program through the National Research Foundation of Korea(NRF) funded by the Ministry of Education (2017R1A6A3A04005963); and H.L.\ was supported by the Leverhulme Trust Early Career Fellowship~ECF-2016-523.}

\date{\today}

\begin{document}

\date{}

\restylefloat{figure} 
\maketitle

\begin{abstract}
A recent result of Condon, Kim, K\"{u}hn and Osthus implies that for any $r\geq (\frac{1}{2}+o(1))n$, an $n$-vertex almost $r$-regular graph $G$ has an approximate decomposition into any collections of $n$-vertex bounded degree trees. In this paper, we prove that a similar result holds for an almost $\alpha n$-regular graph $G$ with any $\alpha>0$ and a collection of bounded degree trees on at most $(1-o(1))n$ vertices if $G$ does not contain large bipartite holes. This result is sharp in the sense that it is necessary to exclude large bipartite holes and we cannot hope for an approximate decomposition into $n$-vertex trees.

  Moreover, this implies that for any $\alpha>0$ and an $n$-vertex almost $\alpha n$-regular graph $G$, with high probability, the randomly perturbed graph $G\cup \mathbf{G}(n,O(\frac{1}{n}))$ has an approximate decomposition into all collections of bounded degree trees of size at most $(1-o(1))n$ simultaneously.
This is the first result considering an approximate decomposition problem in the context of Ramsey-Tur\'an theory and the randomly perturbed graph model.
\end{abstract}

\section{Introduction}

Finding sufficient conditions for the existence of a subgraph of $G$ isomorphic to a specific graph $H$ is a central theme in extremal graph theory. The earliest results of this type are Mantel's theorem~\cite{Man07} and Tur\'an's theorem~\cite{Tur41} stating that an $n$-vertex graph $G$ contains a complete graph $K_r$ on $r$ vertices whenever $G$ contains at least $(1-\frac{1}{r-1})\binom{n}{2}$ edges. Erd\H{o}s-Stone-Simonovits theorem~\cite{ES66,ES46} further generalises this into any small graph $H$.

On the other hand, the nature of problems changes if we consider a `large' graph $H$ whose number of vertices is comparable (or equal) to that of $G$. One important cornerstone in this direction is Dirac's theorem~\cite{Dir52} which shows that whenever we have $\delta(G)\geq \frac{n}{2}$, the $n$-vertex graph $G$ contains a Hamilton cycle.
Koml\'os, S\'ark\"ozy and  Szemer\'edi~\cite{KSS95} proved that the condition of $\delta(G)\geq (\frac{1}{2} + o(1))n$ ensures the containment of every $n$-vertex bounded degree tree as a subgraph, 
and in \cite{KSS01}, they extended this result to the trees with maximum degree $o(\frac{n}{\log{n}})$.
Furthermore, B\"ottcher, Schacht and Taraz \cite{BST09} found a minimum degree condition guaranteeing
the containment of an $n$-vertex graph $H$ with sublinear bandwidth and bounded maximum degree.

Another important research direction in extremal graph theory concerns with decomposition of graphs. 
We say that a collection $\mathcal{H}=\{H_1,\dots, H_s\}$ of graphs \emph{packs} into $G$ if $G$ contains pairwise edge-disjoint copies of $H_1, \dots, H_s$ as a subgraph.
If $\cH$ packs into $G$ and $e(\cH)=e(G)$ (where $e(\cH)=\sum_{H\in \cH} e(H)$), then we say that the graph $G$ has a \emph{decomposition} into $\cH$.
If a packing covers almost all edges of the host graph $G$, then we informally say that $G$ has an \emph{approximate} decomposition. The history of graph decomposition problems dates back to 19th century when Kirkman 
characterised all $n$ such that $K_n$ decomposes into triangles and when Walecki characterised all $n$ such that $K_n$ decomposes into Hamilton cycles.
The latter was extended to Hamilton decompositions of 
regular graphs $G$ of high degree in a seminal work of Csaba, K\"uhn, Lo, Osthus and Treglown \cite{CKOLT}. Yet another generalisation, the famous Oberwolfach conjecture states that
for any $n$-vertex graph $F$ consisting of vertex-disjoint cycles, 
$K_n$ has a decomposition into $F$, except a finitely many values of $n$.
After many partial results, this was finally resolved very recently for all large $n$ by Glock, Joos, Kim, K\"uhn and Osthus \cite{GJKKO18}.

Further famous open problems in the area are the tree packing conjecture of Gy\'arf\'as and Lehel, which says that for any collection  $\mathcal{T}=\{T_1,\dots, T_n\}$ of trees with $|V(T_i)|=i$, the complete graph $K_n$ has a decomposition into $\mathcal{T}$, and Ringel's conjecture which says that for any $(n+1)$-vertex tree $T$, the complete graph $K_{2n+1}$ has a decomposition into $2n+1$ copies of $T$. Lots of research has been done regarding these conjectures, \cite{BHPT16,  FLM15, KKOT:ta, MRS16}. Recently, Joos, Kim, K\"uhn and Osthus \cite{JKKO} proved both conjectures for trees with bounded degree and larger $n$. A key ingredient of their proof is a blow-up lemma for approximate decompositions of $\epsilon$-regular graphs $G$ developed by Kim, K\"uhn, Osthus and Tyomkyn \cite{KKOT:ta}.  Allen, B\"ottcher, Hladk\`y and Piguet~\cite{ABHP} later proved an approximate decomposition result for degenerate graphs with maximum degree $o(\frac{n}{\log{n}})$. Montgomery, Pokrovskiy and Sudakov \cite{MPS18} found an approximate decomposition of $K_{2n+1}$ into any $(1-o(1))n$-vertex tree $T$, proving an approximate version of Ringel's conjecture.  

In \cite{CKKO17}, Condon, Kim, K\"uhn and Osthus determined the degree threshold for an almost regular graph to have an approximate decomposition into a collection $\cH$ of separable graphs with bounded degree. In particular, one corollary of their result is that for any collection $\cT$ of $n$-vertex bounded degree trees, any almost-regular $n$-vertex graph $G$ with degree at least $(\frac{1}{2}+o(1))n$ has an approximate decomposition into $\cT$.

Most of the aforementioned results are sharp as there are graphs which do not satisfy the conditions and do not have a desired subgraph or a desired (approximate)-decomposition. For example, regarding the corollary on approximate tree decomposition, 
a complete balanced bipartite graph $K_{\frac{n}{2},\frac{n}{2}}$ or disjoint union of two complete graphs $2K_{\frac{n}{2}}$ shows that the degrees of $G$ has to be at least $(\frac{1}{2}+o(1))n$ to contain a single copy of an $n$-vertex tree with unbalanced bipartition, let alone an approximate decomposition. However, such examples have very special structures. Hence it is natural to ask how the degree conditions change if we exclude graphs with such special structures.

Another active line of research is to study these changes on the degree conditions when we exclude a large independent set. Balogh, Molla and Sharifzadeh~\cite{BMSh16} initiated this by proving that if an $n$-vertex $G$ does not contain any linear-sized independent set and $\delta(G) \geq (\frac{1}{2}+o(1))n$, then $G$ contains a triangle-factor. This weakens the bound $\delta(G)\geq \frac{2}{3}n$ from the Corr\'adi-Hajnal theorem~\cite{CH63}. Nenadov and Pehova \cite{NP18} further generalised this into a $K_r$-factor.

However, excluding large independent sets is not sufficient to guarantee a large connected subgraph, e.g. $2K_{\frac{n}{2}}$ does not contain an independent set of size three, and clearly it does not contain any tree with more than $\frac{n}{2}$ vertices.
This example suggests that it is necessary to exclude large bipartite holes, rather than independent sets.
 An \emph{$(s,t)$-bipartite hole} in a graph $G$ consists of two disjoint vertex sets $S,T\subseteq V(G)$ with $|S|=s, |T|=t$ such that there are no edges between $S$ and $T$ in $G$. The \emph{bi-independence number} $\widetilde{\alpha}(G)$ of a graph $G$ denotes the largest number $r$ such that $G$ contains an $(s,t)$-bipartite hole for every pair of non-negative integers $s$ and $t$ with $s+t=r$. Note that $\widetilde{\alpha}(G)\leq r$ implies that there is at least one edge between any two disjoint vertex sets of size $r$, i.e. $K_{r,r}\nsubseteq \overline{G}$.
McDiarmid and Yolov \cite{MY17} proved the existence of Hamilton cycle on a graph $G$ satisfying $\delta(G)\geq \widetilde{\alpha}(G)$.

Our main theorem states that if $G$ has sublinear bi-independence number and $\cT$ consists of bounded degree trees with at most $(1-o(1))n$ vertices, then the degree threshold for an approximate tree-decomposition of Condon, Kim, K\"uhn and Osthus can be significantly lowered. 
There is an obvious analogy between this theme and the Ramsey-Tur\'{a}n theory in which one studies Tur\'an type problmes for graphs with sublinear independence number. See e.g. \cite{SS01} for more of Ramsey-Tur\'{a}n theory.
Here we replace a Tur\'an-type conclusion with one along the lines of approximate decomposition of $G$ into large graphs.

\begin{theorem}\label{thm:main}
For all $\Delta \in \mathbb{N}$, $0< \alpha, \nu < 1$, there exist $\xi, \eta >0$ and  $n_0 \in \mathbb{N}$ such that  the following holds for all $n \ge n_0$. Suppose that $G$ is an $n$-vertex graph such that $d_{G}(v) = (\alpha\pm \xi)n$ for all vertices $v\in V(G)$ except at most $\xi n$ vertices and $\widetilde{\alpha}(G)\leq \eta n$.
Then any collection $\mathcal{T}$ of trees $T$ satisfying the following conditions packs into $G$.
\begin{enumerate}[label=(\roman*)]
\item $|T| \le (1-\nu) n$ and $\Delta(T) \le \Delta$ for all $T \in \mathcal{T},$
\item $e(\mathcal{T}) \le (1-\nu)e(G).$
\end{enumerate}
\end{theorem}

Note that by considering a collection of paths of length $(1-o(1))n$, it is easy to see that the almost regular degree condition on $G$ is necessary. 
Theorem~\ref{thm:main} is sharp in several point of views.
First, the condition on $\widetilde{\alpha}(G)$ is necessary as embedding even a single copy of $(1-o(1))n$-vertex tree into $2K_{\frac{n}{2}}$ is impossible. Hence, $\widetilde{\alpha}(G)$ is the correct parameter to consider. Second, the trees in $\cT$ having at most $(1-o(1))n$ vertices is also best possible. To see this, we consider a copy of slightly unbalanced complete  bipartite graph with parts $X_1$ of size $\frac{1}{2}(1-\xi)n$ and $X_2$ of size $\frac{1}{2}(1+\xi)n$. We put a copy of random graphs  $\bG(\frac{(1-\xi)n}{2}, \frac{\xi}{2}), \bG(\frac{(1+\xi)n}{2}, \frac{\xi}{2})$ in  $X_1$ and $X_2$, respectively. Let $G$ be the resulting graph and let $\cT$ be a collection of $\frac{1}{2}(1-\nu)n$ copies of $n$-vertex paths.
Then $G$ satisfies all the conditions in Theorem~\ref{thm:main} except that the trees are now spanning. 
Even more, it satisfies a stronger condition that $\widetilde{\alpha}(G) = O(\log{n})$ and all vertices in $G$ has degree $(\frac{1}{2}\pm \xi)n$.
However, as each path has the unique bipartition which is almost balanced, each path uses at least $\xi n-2$ edges inside the bigger part $X_2$. Thus, we need at least $(\xi n -2)\frac{(1-\nu)n}{2}$ edges inside the bigger part in order to pack $\cT$ into $G$. Since  $\bG(\frac{(1+\xi)n}{2}, \frac{\xi}{2})$ only contains at most $\frac{\xi n^{2}}{3}$ edges, $\cT$ does not pack into $G$ if $\nu< \frac{1}{3}$.
 
As the last example contains two vertices $u,v$ with degree difference at least $\frac{\xi n}{2}$,
one might speculate that it is plausible to obtain a packing of spanning trees into $G$ if one additionally assume that $G$ is
much closer to being regular. However, the following example shows that we still need more conditions. Consider a graph $G$ obtained from $K_{\frac{n}{2},\frac{n}{2}}$ by putting a copy of  $\bG(\frac{n}{2}, o(\frac{\log{n}}{n}))$ in each part $X_1, X_2$, respectively. It is easy to see that $\widetilde{\alpha}(G)=o( \frac{n}{\sqrt{\log{n}}})$.
Let $\cT$ be the collection of $\frac{1}{4}(1-\nu)n$ copies of $n$-vertex complete ternary tree $T$ of height $O(\log{n})$.
Csaba, Levitt, Nagy-Gy\H{o}rgy and Szemer\'edi \cite{CLNS10} showed that any embedding of such complete ternary tree $T$ must use at least $\frac{1}{17}\log{n}$ non-crossing edges inside parts $X_i$ of $G$.
Thus, we need at least $\frac{1}{68}(1-\nu)n\log{n}$ non-crossing edges to obtain a packing of $\cT$ into $G$. However, $G$ contains at most $o(n\log{n})$ non-crossing edges. Hence, this shows that it is necessary that the trees in $\cT$ have at most $(1-o(1))n$ vertices. It is not difficult to modify the above example to obtain a regular graph $G$ (rather than just close to being regular) with $\widetilde{\alpha}(G)=o(n)$ which does not admit an approximate decomposition into complete ternary trees.

Our theorem has a corollary in randomly perturbed graph model which combines extremal and probabilistic aspects in one graph model.
Bohman, Frieze and Martin \cite{BFM03} introduced the concept of randomly perturbed graph model by proving that given any fixed $\alpha >0$ there exists a constant $C$ such that
for any $n$-vertex graph $G$ with $\delta(G)\geq \alpha n$, the graph $G\cup \bG(n,\frac{C}{n})$ contains a Hamilton cycle with high probability.
This sparks numerous research see e.g.~\cite{BTW17,BHKM18, BFKM04, BHKMPP18, BMPP18, HZ18, JK18, KKS16, KKS17,  KST04,MM18}.

The following corollary is a direct consequence of Theorem~\ref{thm:main}. It is easy to see that for large constant $C$, the random graph $\bG(n,\frac{C}{n})$ does not contain a $(C^{-1/3}n, C^{-1/3}n)$-bipartite hole, hence $\widetilde{\alpha}(G\cup \bG(n,\frac{C}{n})) \leq 2 C^{-1/3} n$.

\begin{corollary}\label{thm:near spanning}
For all $\Delta \in \mathbb{N}$, $0< \alpha, \nu < 1$, there exist $\xi_0 >0$ and  $n_0, C \in \mathbb{N}$ such that the following holds for all $\xi \le \xi_0$ and $n \ge n_0$. Suppose that $G$ is an $n$-vertex graph such that $d_{G}(v) = (\alpha\pm \xi)n$ for all vertices $v\in V(G)$ except at most $\xi n$ vertices
and $\mathbf{G}(n,\frac{C}{n})$ is a binomial random graph on the vertex set $V(G)$.
Then  the following holds with high probability.
Any collection $\mathcal{T}$ of trees $T$ satisfying the following conditions packs into $G\cup \bG(n,\frac{C}{n})$.
\begin{enumerate}[label=(\roman*)]
\item $|T| \le (1-\nu) n$ and $\Delta(T) \le \Delta$ for all $T \in \mathcal{T},$
\item $e(\mathcal{T}) \le (1-\nu)e(G).$
\end{enumerate}
\end{corollary}

Note that the above statement is universal in the sense that with high probability, this holds for every collection $\cT$ simultaneously. Corollary~\ref{thm:near spanning} is  sharp in the following senses. By considering a disconnected graph $G$, it is easy to see that the probability $O(\frac{1}{n})$ is best possible.
Also the trees having size $(1-o(1))n$ is best possible. 
The above first example obtained from slightly unbalanced complete bipartite graph show that we need $\mathbf{G}(n,\Omega(1))$ in order to obtain an approximate decomposition of almost $(\alpha\pm \xi)n$-regular graphs into spanning trees with bounded maximum degree, and the second example with complete ternary trees shows that at least $\mathbf{G}(n, \Omega(\frac{\log{n}}{n}))$ is required for 
obtaining an approximate decomposition of regular graphs into spanning trees.
Motivated by this second example, we ask the following question.
\begin{question}\label{question 1}
Determine the optimal function $f$ satisfying the following.
For given $\alpha, \nu>0$, if $G$ is an $n$-vertex $\lfloor \alpha n \rfloor$-regular graph and $\widetilde{\alpha}(G)\leq o(f(n))$.
Let $\mathcal{T}$ be a collection trees $T$ satisfying the following conditions.
\begin{enumerate}[label=(\roman*)]
\item $|T| \le n$ and $\Delta(T) \le \Delta$ for all $T \in \mathcal{T},$
\item $e(\mathcal{T}) \le (1-\nu)e(G).$
\end{enumerate}
Then $\cT$ pack into $G$.
\end{question}
We can consider the same question of finding the optimal function $g(n)$ by replacing the graph $G$ with  $G'\cup \bG(n, \omega(g(n)))$ and an arbitrary $n$-vertex $\lfloor \alpha n \rfloor$-regular graph $G'$.
Since the regularity lemma does not distinguish between an $(\alpha\pm o(1))n$-regular graph and an $\alpha n$-regular graph,
the example we obtained from $K_{(1-\xi)n,(1+\xi)n}$ shows that Question~\ref{question 1} may not be proved by the approach in this paper which is based on the regularity lemma.

Our theorem also has further applications on tree packing conjectures, such as Ringel's conjecture in the setting of almost regular graphs.
It implies that if $\alpha>0$ and $G$ is an almost $\alpha n$-regular $(2n+1)$-vertex graph with $\widetilde{\alpha}(G) =o(n)$, and $T$ is an $n$-vertex tree with bounded maximum degree, then $G$ has an approximate decomposition into $(1-o(1))\alpha n$ copies of $T$. Same statement also holds for $G\cup \bG(n,\frac{C}{n})$ with any almost $\alpha n$-regular $(2n+1)$-vertex graph $G$.

\section{Preliminaries}\label{sec: prelim}

Denote $[t] := \{1, \dots, t\}.$  If we claim that a result holds for $1/n \ll a \ll b \ll1$, this mean  that there exist non-decreasing functions $f:(0,1] \rightarrow (0,1] $ and $g:(0,1] \rightarrow (0,1]$ such that the result holds for all $0 \le a, b \le 1$ and all $n \in \mathbb{N}$ with $a \le f(b)$ and $1/n \le g(a)$. 
 We may omit floors and ceilings when they are not essential.
In this paper, \emph{graphs} are simple undirected finite graphs and  \emph{multigraphs} are graphs with potentially parallel edges without loops. 

Given  collection of trees $\cT$, denote by $|\cT|$ the number of trees in $\cT$ and $e(\cT):= \sum_{T\in \cT} e(T)$.
Let $G =(V,E)$ be a graph and $A, B \subseteq V(G)$ satisfying $A \cap B = \emptyset$. 
Denote by $E_G(A,B)$  the set of edges in $G$ between $A$ and $B$.  
Let $e_G(A,B):= |E_G(A,B)|$.
For sets $X, A\subseteq V(G)$, we define $N_{G,A}(X) := \{ w \in A : uw \in E(G) \text{ for all } u\in X \}$. In particular, we have $N_{G,A}(\emptyset)=A$, and let $N_{G}(X):=N_{G,V(G)}(X)$. Let $d_{G,A}(X) = |N_{G,A}(X)|$.
We write $d_{G,A}(v_1,\dots, v_i)$ for $d_{G,A}(\{v_1,\dots, v_i\})$.
  Denoted by  $N_G^d(X) \subseteq V(G) \backslash X$  the set of vertices of distance at most $d$ from a vertex in a set $X \subseteq V(G)$.  Note that, in this definition, $N_{G}(X)$ and $N_{G}^{1}(X)$ are in general different for $|X|>1$.  For a tree $T$ and a vertex $x$, let $(A_T(x), B_T(x))$  be the unique vertex partition into two independent sets satisfying $x\in A_T(x)$. Denote by $G\setminus A$ the induced subgraph on $V(G)\setminus A$, and by $G-E$ the spanning subgraph with edge set $E(G)-E$, where $A\subseteq V(G)$ and $E\subseteq E(G)$.
For a graph $G$ and two disjoint vertex subsets $A$ and $B$, the {\it density} of $(A,B)$ is defined as  $$\den_G(A,B): = \frac{e_G(A,B)}{|A||B|}.$$
For a rooted tree $(T,r)$ with the root $r$,  let $T(u)$ be the subtree of $T$ consisting of all vertices $v$ such that the path between $r$ and $v$ contains $u$. For a vertex $x\in V(T)$, denoted by $a_T(x)$  the parent of $x$. Denoted by 
$D_{T}^k(x)$  the set of all descendents $y$ of $x$ with distance exactly $k$ in the tree $T$, and by $D_{T}^{\leq k}(x)$ be the set of descendents $y$ of $x$ with distance at most $k$.  We write $D_{T}(x) := D_{T}^{1}(x)$.
 For two functions $\phi:A\rightarrow B$ and $\phi': A'\rightarrow B'$ with $A\cap A'=\emptyset$, we define $\phi\cup \phi'$ as a function from $A\cup A'$ to $B\cup B'$ such that for each $x\in A\cup A'$, $$(\phi\cup\phi')(x) := \left\{ \begin{array}{ll}
\phi(x) & \text{ if } x\in A \\
\phi'(x) & \text{ if } x\in A'.
\end{array}\right.$$

%
We will use well-known Chernoff's inequality and Azuma's inequality. As our applications are very simple and standard, we will omit the detailed computation. See \cite{Azu67, Hoe63, JLR00} for the statements of Chernoff's inequality and Azuma's inequality.
The concept of $(\epsilon, d)$-\emph{regularity} and \emph{Szemer\'edi's regularity lemma} will be useful for us.
A bipartite graph $G$ with vertex partition $(A,B)$ is  \emph{$(\epsilon,d)$}-\emph{regular} if
for all sets $A'\subseteq A$, $B'\subseteq B$ with  $|A'|\geq \epsilon |A|$, $|B'|\geq \epsilon |B|$,
we have
\begin{align*}
	| \den_G(A',B')- d| < \epsilon.
\end{align*}
A bipartite graph $G$   is \emph{$\epsilon$-regular}
if $G$ is $(\epsilon,d)$-regular for some $d$.
Additionally, a bipartite graph $G$ is  $(\epsilon,d)$\emph{-super-regular} if $G$ is $(\epsilon,d)$-regular with $d_{G}(a)= (d\pm \epsilon)|B|$ for $a\in A$ and $d_{G}(b)= (d\pm \epsilon)|A|$ for $b\in B$. The following three well-known lemmas will be useful when we modify a given $\epsilon$-regular partition.


\begin{proposition}\label{prop: reg degree}
Let $0<\epsilon \ll d \leq 1$. Suppose that $G$ is an $(\epsilon,d)$-regular bipartite graph with vertex partition $(A,B)$. Let $B'\subseteq B$ be a set with $|B'| \geq \epsilon^{1/3} |B|$. Then $A$ contains at least $(1-\epsilon^{1/2})|A|$ vertices $u$ satisfying $d_{G,B'}(u)\geq (d-\epsilon^{1/2})|B'|$.
\end{proposition}

\begin{proposition}\label{prop: super subgraph}
Let $0<\epsilon \ll d \leq 1$. Suppose that $G$ is an $(\epsilon,d)$-regular bipartite graph with vertex partition $(A,B)$. Then, there exists sets $A'\subseteq A, B'\subseteq B$ with $|A'|\geq (1-2\epsilon) |A|$ and $|B'|\geq (1-2\epsilon)|B|$ 
such that $G[A',B']$ is an $(3\epsilon,d)$-super-regular bipartite graph with vertex partition $(A',B')$.
\end{proposition}

\begin{proposition}\label{prop: subgraph reg}
Let $0<\epsilon \ll d \leq 1$. Suppose that $G$ is an $(\epsilon,d)$-regular bipartite graph with vertex partition $(A,B)$. Let $E$ be a set of edges with $|E|\leq \epsilon^{10} |A||B|$.
Then $G-E$ is $(\epsilon^{1/2}, d)$-regular.
\end{proposition}


The following two lemmas will be useful for finding some edge/vertex partition of graphs. We omit the proofs as they easily follow from a standard random splitting argument.

\begin{lemma}\label{lem: reg decomp}
Let $0 < 1/n  \ll  \epsilon \ll d, 1/s, 1/t<1$.
Suppose that $G$ is an $(\epsilon,d)$-regular bipartite graph with vertex partition $(A,B)$ satisfying $|A|,|B|\geq n$. Let $p_1,\dots, p_s \in [0,1]$ be values such that $p_1+\dots + p_s \leq 1$.
Then there exist edge-disjoint spanning subgraphs $G_1,\dots, G_s$ of $G$ such that $G_i[A,B]$ is $(2\epsilon,dp_i)$-regular for each $i\in [s]$.
\end{lemma}

\begin{proposition}\label{prop: reg partition}
Let $0<1/n\ll \epsilon \leq d \leq 1$. Suppose that $G$ is an $(\epsilon,d)$-super-regular bipartite graph with vertex partition $(A,B)$ satisfying $|A|,|B|\leq n$.
Let $a_1,a_2, b_1,b_2 \in \mathbb{N}$ be numbers such that $a_1, a_2,b_1,b_2\geq \epsilon n$, $a_1+a_2=|A|$, and $ b_1+b_2=|B|$. Then there exists a partition $A_1,A_2$ of $A$ and $B_1, B_2$ of $B$ such that 
for any $i,j\in [2]$, we have $|A_i|=a_i, |B_j|=b_j$ and
the graph $G[A_i,B_j]$ is $(\epsilon^{1/2},d)$-super-regular.
\end{proposition}

The following is a version of well-known Szemer\'edi's regularity lemma.

\begin{lemma} [Szemer\'edi's regularity lemma]
\label{Szemeredi} 
Let $M, M', n \in \mathbb{N}$  and $0<1/n \ll 1/M \ll 1/M', \epsilon \leq 1$. Then for any $n$-vertex graph $G$, there exists a partition of $V(G)$ into $V_0, V_1, ..., V_r$ and a spanning subgraph $G' \subseteq G$ satisfying the following: \begin{enumerate}[label=(\roman*)]
\item $M' \le r \le M,$
\item $|V_0| \le \epsilon n,$
\item $|V_i| = |V_j|$ for all $i, j \in [r],$
\item $d_{G'}(v) > d_G(v)  - (d + \epsilon)n$ for all $ v \in V (G),$
\item $e(G'[V_i]) = 0$ for all $ i \in [r],$
\item For all $i, j$ with $1 \le i \le j \le r$, the graph $G' [V_i, V_j ]$ is either empty or $(\epsilon, d_{i,j})$-regular for some $d_{i,j} \in [d,1]$.
\end{enumerate}

\end{lemma}

The following two lemmas will be useful to utilise the assumption on bi-independence number.
\begin{lemma}\label{lem: highly connecting sets}
Let $0<1/n\ll \eta \ll 1$. Suppose $G$ is an $n$-vertex graph with $\widetilde{\alpha}(G)\leq \eta n$. If $W$ and $W'$ are two (not necessarily  disjoint) subsets of $V(G)$ with size at least $2\eta^{1/3} n$, then all but at most $2\eta n$ vertices $w$ in $W$ satisfies $d_{G,W'}(w)\geq \eta^{-1/2}$.
\end{lemma}
\begin{proof}
Suppose that the lemma does not hold, then there exists a set $Z\subseteq W$ of exactly $2\eta n$ vertices $w$ satisfying $d_{G,W'}(w) < \eta^{-1/2}$. Then the set $Z':= W' \setminus  (Z\cup N^{1}_{G}(Z))$ contains at least 
$\eta^{1/3} n - (\eta^{-1/2}+1)\cdot 2\eta n \geq 2\eta n$ vertices.
Hence $e_G(Z,Z') =0$, contradicting $\widetilde{\alpha}(G)\leq \eta n$. This proves the lemma.
\end{proof}

\begin{lemma}\label{lem: subgraph biindependence}
Let $0<1/n\ll \eta \ll \xi \ll 1$. Suppose that $G$ is an $n$-vertex graph with $\widetilde{\alpha}(G)\leq \eta n$. Then there exists a spanning subgraph $H$ of $G$ with $\widetilde{\alpha}(H)\leq 2\eta^{1/3} n$ and $\Delta(H)\leq \xi n$.
\end{lemma}
\begin{proof}
For each edge $e$ of $G$, we include it in $H$ independently at random with probability $\xi/2$. A standard application of Chernoff's inequality implies that, with probability at least $0.9$, we have $d_{H}(v) \leq \xi n$ for all $v\in V(G)$.
We consider two disjoint sets $A,B \subseteq V(G)$ with $|A|=|B|=\eta^{1/3}n$. 
By lemma~\ref{lem: highly connecting sets}, we have
$$\mathbb{E}[e_H(A,B)]= \frac{\xi}{2} \cdot e_G(A,B) \geq \frac{\xi}{2} \cdot \eta^{-1/2} \cdot (\eta^{1/3}n-2\eta n) \geq \eta^{-1/7}n.$$
A standard application of Chernoff's inequality implies that with probability at least $1 - \exp(-\eta^{-1/10}n)$, we have $e_H(A,B) > 0$.
By a union bound, with probability at least $1 - 2^{2n}\cdot \exp(-\eta^{-1/10}n)\geq 0.9$, we have that $e_H(A,B)>0$ for all disjoint sets $A,B \subseteq V(G)$ with $|A|=|B|=\eta^{1/3}n$. This implies that $\widetilde{\alpha}(H)\leq 2\eta^{1/3}n$. Hence, with probability at least $0.8$, $H$ has the desired properties.
\end{proof}

The following proposition from \cite{JKKO} provides a useful partition of a tree. 

\begin{proposition}\cite{JKKO}\label{prop: tree partition}
Let $n,  \Delta \in \mathbb{N} \backslash \{1\}$ and $n\geq t\geq 1$. 
Then for any rooted tree $(T,r)$ on $n$ vertices with $\Delta(T)\leq \Delta$, there exists a collection $\mathcal{S}$ of pairwise vertex-disjoint rooted subtrees such that the following holds.
\begin{enumerate}
\item $S \subseteq T(x)$ for every $(S, x) \in \mathcal{S}.$
\item $t \le |S| \le 2\Delta t$ for every $(S, x) \in \mathcal{S}.$
\item $\bigcup_{(S,x) \in \mathcal{S}} V(S) = V(T).$
\end{enumerate} 
\end{proposition}

The following two results are from \cite{KKOT:ta}.
The first lemma is a special case of Lemma~7.1 in \cite{KKOT:ta}.

\begin{lemma} \cite{KKOT:ta}
\label{Pack}
Let $n, \Delta, k, s \in \mathbb{N}$ and $0<1/n \ll \epsilon, 1/k \ll 1$.
 Suppose that $0< \zeta < 1$ with $s^{2/3} \le \zeta k$.
 Let $G$ be an $2n$-vertex balanced complete bipartite graph with the vertex partition $V_1$ and $V_2$.
Suppose  that the graph $L_j$  is a subgraph of $G$ with the vertex partition $X_{1}^{j}$ and $X_{2}^j$ such that $\Delta (L_j) \le \Delta$ for each $j \in [s]$. Suppose that we have 
$\sum_{j=1}^{s} e(L_j) \leq (1 - 2\zeta) kn$
and sets $W_i^j \subseteq X_i^j$ satisfies $|W_i^j| \le \epsilon n$ for all $j \in [s]$ and $i\in [2]$.
Then there exists a $k$-regular spanning subgraph $H$ of $G$ and a function $\phi$ which packs $\cL:=\{L_1, \dots, L_s\}$ into $H$ such that $\phi(X_i^j) \subseteq V_i$ and   $\phi(W_i^j) \cap \phi(W_i^{j'}) = \emptyset$ for distinct $j, j' \in [s]$ and $i \in [2]$.
\end{lemma}

\begin{theorem} [Blow-up lemma for approximate decompositions \cite{KKOT:ta}]
\label{Blowup}
Let $n,k \in \mathbb{N} $ and $0<1/n \ll \epsilon \ll \alpha, d, d_0, 1/k \le 1$. Let $s \in \mathbb{N}$ be a number such that $s \le (1-\frac{\alpha}{2})\frac{d n}{k}$.
Suppose  that the following properties hold.

  \begin{enumerate}[label=\text{\rm (A\arabic*)}$_{\ref{Blowup}}$]
\item\label{Blowup A1} $G$ is a $(\epsilon, d)$-super-regular graph with the vertex partition $V_1$ and $V_2$.
\item\label{Blowup A2} $\mathcal{H} = \{H_1, \dots ,H_s\}$, where each $H_i$ is an $k$-regular bipartite graph with the vertex partition $X_1$ and $X_2$.
\item\label{Blowup A3}  For all $j \in [s]$ and $ i \in [2],$ there is a set $Y_i^j \subseteq X_i$ with $|Y_i^j| \le \epsilon n$ and for each $y\in Y_i^j$, there is a set $A_y^i \subseteq V_i $ with $|A_y^i| \ge d_0 n.$
\item\label{Blowup A4}  $\Gamma$ is a graph with $V(\Gamma) \subseteq [s] \times V(G)$ and $\Delta(\Gamma) \le (1-\alpha)d_0n$ such that for each $(i,x) \in V(\Gamma)$ and $i' \in [s]$, we have $|\{x' \in V(G): (i',x') \in N_\Gamma(i,x) \}| \le k^2.$
Moreover, for each $i \in [s]$ and $j \in [2]$, we have $|\{ (i,x) \in V(\Gamma) : x \in X_j\}| \le \epsilon n.$
\end{enumerate}
Then there exists a function $\phi$ packing $\mathcal{H}$ into $G$ such that for all $j,j' \in [s]$ and $i \in [2]$\begin{enumerate}[label=\text{\rm (B\arabic*)}$_{\ref{Blowup}}$]
\item\label{Blowup B1}  $\phi(V(H_j) \cap X_i) \subseteq V_i.$ 
\item\label{Blowup B2}  $\phi(y) \in A_y^i$ for all $y \in \bigcup_{j\in [s]} Y_i^j$. 
\item\label{Blowup B3}  For all $(j,x)(j',y) \in E(\Gamma),$ we have  $\phi(x) \ne \phi(y).$
\end{enumerate}
\end{theorem}

\section{Proof of Theorem \ref{thm:main}}\label{sec: main proof}

In this section, we prove our main theorem assuming the following lemma which will be proved in Section~\ref{sec: main lemma}. This lemma states that if $G$ admits a certain $\epsilon$-super-regularity partition and $\widetilde{\alpha}(G) = o(|V(G)|)$, then we can find an approximate decomposition of $G$ into arbitrary bounded degree $(1-o(1))|V(G)|$-vertex trees. Here, $G$ consists of sets $V_{s,i}$ and $U_{s,i}$ which form an $\epsilon$-super-regular matching structure. The reduced graph for this $\epsilon$-regular partition is not connected, but the condition on $\widetilde{\alpha}(G[U])$ provides a connection
necessary to embed trees.

\begin{lemma}
\label{lem: embed one}
Suppose $0<1/n \ll 1/M \ll 1/r, \epsilon \ll d \ll \nu, 1/\Delta <1$.
Let $G$ be a graph with a vertex partition $U \cup V$ and let 
$\{ U_{s,i} : (s,i)\in [r]\times [2]\}$ be a partition of $U$ and 
$\{V_{s,i}: (s,i) \in [r]\times [2]\}$ be a partition of $V$.
Let $\mathcal{T}$ be a collection of trees on at most $2(1-\nu)rn$ vertices with maximum degree at most $\Delta$.
Assume that the following properties hold. 
\begin{enumerate}[label=\text{\rm (A\arabic*)}$_{\ref{lem: embed one}}$]
\item \label{lem: embed one 1} For each $s\in [r]$, each of four graphs 
$G[V_{s,1},V_{s,2}]$, $G[U_{s,1},U_{s,2}]$, $G[U_{s,1}, V_{s,2}]$ and $G[V_{s,1}, U_{s,2}]$ are $(\epsilon, d)$-super-regular with $|V_{s,1}| = |V_{s,2}| = n$ and $|U_{s,1}|, |U_{s,2}| \geq \epsilon n$. 

\item\label{lem: embed one 2}  $\widetilde{\alpha}(G[U])\leq M^{-6}n$.

\item\label{lem: embed one 3}   $e(\mathcal{T}) \le (1-\nu)rdn^2.$
\end{enumerate}
Then there exists a map $\phi$ packing $\mathcal{T}$ into $G$ such that  $d_{\phi(\mathcal{T})}(u) \leq  \Delta M$ for each $u\in U$.
\end{lemma}

We start the proof of Theorem \ref{thm:main}.
For given $\nu$ and $\alpha$, we choose constants $n_0, \eta, \xi,\epsilon,  t$ such that 
\begin{align}\label{eq: hierarchy}
0 < \frac{1}{n_0} \ll \eta  \ll \epsilon  \ll  \xi \ll \frac{1}{t} \ll \nu, \alpha <1.
\end{align}
Let $n \ge n_0$. 
By deleting exactly $\xi n$ vertices with degree furthest from $\alpha n$, we can assume that $G$ is an $(1-\xi)n$-vertex graph such that every vertex $v \in V(G)$ satisfies $d_{G}(v) = (\alpha\pm 2\xi) n$. By Lemma~\ref{lem: subgraph biindependence}, we can find a spanning subgraph $H$ of $G$ with $\widetilde{\alpha}(H)\leq 2\eta^{1/3}n$ and $\Delta(H)\leq \xi n$.
By replacing $G$ with $G-E(H)$, assume that $G$ and $H$ are edge-disjoint graphs and 
\begin{align}\label{eq: basic assumption}
\widetilde{\alpha}(H) \leq 2\eta^{1/3} n \text{ and }d_{G}(v) = (\alpha \pm 3\xi)n \text{ for all $v \in V(G)$.}
\end{align}
Suppose that $\mathcal{T}$ is a collection of trees satisfying $(i)$ and $(ii)$. Now we aim to construct (not disjoint) sets $U^1,\dots, U^\kappa, V^1,\dots, V^{\kappa}$ and edge-disjoint subgraphs $G_1, \dots, G_{\kappa}$ of $G$.  We will also
 partition $\mathcal{T}$ into $\kappa$ subcollections of trees $\mathcal{T}_1, \dots, \mathcal{T}_\kappa$, and pack the trees of $\mathcal{T}_i$ into $G_i \cup H[U^i].$ \vspace{0.2cm}

\noindent {\bf Step 1. Partitioning $G$.}
First, we will partition $G$ into graphs with 
appropriate structure which are suitable for applications of Lemma \ref{lem: embed one}. 
We apply Szemer\'edi's regularity lemma (Lemma \ref{Szemeredi})  with $(\epsilon, \frac{1}{t}, \epsilon^{-1}, \eta^{-1/100})$ playing the role of $(\epsilon, d, M', M)$ to obtain a partition $V_0', \dots, V_{r}'$ of $V(G)$ and a spanning subgraph $G' \subseteq G$ satisfying the following.

\begin{enumerate}[label=\text{\rm (R\arabic*)}]
\item \label{R1} $\eta^{1/100} \leq \frac{1}{r} \leq \epsilon,$
\item \label{R2} $|V_0'| \le \epsilon n,$
\item \label{R3} $|V_i'| = |V_j'| = (1\pm \epsilon)\frac{n}{r}$ for all $i, j \in [r],$
\item \label{R4}   $d_{G'}(v) > d_G(v)  - \frac{2 n}{t}$ for all  $ v \in V (G)$
\item \label{R5} $e(G'[V'_i]) = 0$ for all $ i \in [r],$
\item \label{R6} for any $i,j\in [r]$, the graph $G' [V'_i, V'_j ]$ is either empty or $(\epsilon, d_{i,j})$-regular for some $d_{i,j} = d_{j,i} \in [\frac{1}{t},1]$.
\end{enumerate}
Let $R$ be a reduced graph with
$$V(R)= [r] \text{ and } E(R):=\{ ij: e_{G'}(V'_i,V'_j)>0 \}.$$
As $ij\in E(R)$ if and only if $G'[V'_i,V'_j]$ is $(\epsilon, d_{i,j})$-regular with $d_{i,j}\geq 1/t$, for each $i\in [r]$, we have
\begin{eqnarray}\label{eq: sum dij}
\sum_{j\in N_R(i)} d_{i,j} \hspace{-0.2cm} &=& \hspace{-0.2cm} \sum_{j \in N_{R}(i)} \left(\frac{e_{G'}(V'_i,V'_j)}{|V'_i||V'_j|} \pm \epsilon\right) \stackrel{\ref{R5}}{=} \frac{\sum_{v \in V'_i} d_{G', V(G)\setminus V_0}(v) }{|V'_i|^2} \pm \epsilon r \nonumber \\
 \hspace{-0.2cm}&\stackrel{\ref{R2},\ref{R4}}{=}&  \hspace{-0.2cm}  \frac{1}{|V'_i|^2} \sum_{v\in V'_i} \left(d_{G}(v) \pm \frac{3n}{t}\right) \pm \epsilon r 
\stackrel{\eqref{eq: basic assumption}}{=} \frac{ ( \alpha \pm 3\xi \pm \frac{3}{t}) n }{|V'_i| } \pm \epsilon r
 \stackrel{\ref{R3}}{=} \left(\alpha \pm \frac{5}{t}\right) r 
\end{eqnarray}

Now we will find edge-disjoint subgraphs of $G'$ each of which admits $\epsilon$-regular matching structure.
For each $ij\in E(R)$, letting $t_{i,j}:=  \lfloor d_{i,j}\cdot t\rfloor $, we use Lemma~\ref{lem: reg decomp}  with $G[V'_i, V'_{j}], t_{i,j} ,  \frac{1}{d_{i,j}t}$ playing the roles of $G, s, p_1=\dots= p_{s}$, respectively, to obtain edge-disjoint subgraphs $E^1_{i,j}, \dots, E^{t_{i,j}}_{i,j} $
of $G'[V'_i,V'_{j}]$. For each $ij\in E(R)$ and $\ell\in [t_{i,j}]$, 
\begin{align}\label{eq: regularity1}
\text{\em $E^\ell_{i,j}$ is $\left(2\epsilon,\frac{1}{t}\right)$-regular.}
\end{align}
We will take an appropriate unions of these graphs $E^{\ell}_{i,j}$ to form $\epsilon$-regular matching structures.
Let $R^*$ be a multi-graph obtained by replacing each edge $ij$ of $R$ with $t_{i,j}$ edges $e_{i,j}^1,\dots, e_{i,j}^{t_{i,j}}$ between the vertices $i$ and $j$.
Let $\Phi$ be a map from $E(R^*)$ to $\{E_{i,j}^{\ell}: \ell\in [t_{i,j}], ij\in E(R)\}$ such that $\Phi(e_{i,j}^{\ell}) = E_{i,j}^{\ell}$.
For each $i \in [r]$, we have
\begin{eqnarray}\label{eq: R* regular}
d_{R^*}(i) &=&\sum_{j \in N_{R}(i) } t_{i,j}  = \sum_{j\in N_{R}(i)} (d_{i,j} t \pm 1) 
\stackrel{\eqref{eq: sum dij}}{=}
\left(\alpha \pm \frac{6}{t} \right)tr.
\end{eqnarray}
Let
\begin{align}\label{eq: kappa def}
\kappa : = (\alpha - t^{-1/3}) tr.
\end{align}
By applying Vizing's theorem to $R^*$, we obtain $(\alpha+\frac{6}{t})tr+t$ edge-disjoint (possibly empty) matchings covering all edges of $R^*$.
By \eqref{eq: R* regular} and the pigeonhole principle, at least $\kappa$ matchings contain at least $(1-t^{-1/3})\frac{r}{2}$ edges. Let $M_1,\dots, M_{\kappa}$ be edge-disjoint
matchings of $R^*$ of size at least $(1-t^{-1/3}) \frac{r}{2}$, thus for each $i\in [\kappa]$,
\begin{align}\label{eq: matching size}
(1-t^{-1/3}) \frac{r}{2} \leq |E(M_k)| \leq \frac{r}{2}.
\end{align}

For $k\in [\kappa]$, we write $ij\in E(M_k)$ if $M_k$ contains one of $e_{i,j}^1,\dots, e_{i,j}^{t_{i,j}}$.
For each $k\in [\kappa]$, we define a graph $G_k$ with 
$$V(G_k) := \bigcup_{i\in V(M_k)} V'_i \enspace \text{and} \enspace E(G_k):= \bigcup_{ e\in E(M_k)} \Phi(e).$$

For each $k\in [\kappa]$ and $ij\in E(M_k)$, apply Proposition~\ref{prop: super subgraph} to obtain sets $W^k_i\subseteq V'_i$ and $W^k_j\subseteq V'_j$ such that both $W^k_i$ and $W^k_j$ have size $(1\pm 3\epsilon)\frac{n}{r}$ and 
$G_k[W^k_i, W^k_j]$ is $(3\epsilon,\frac{1}{t})$-super-regular.
We further apply Proposition~\ref{prop: reg partition} to $G_k[W^k_i, W^k_j]$ for each $ij\in E(M_k)$ with $a_1=b_1= n_{\bullet}:=\frac{(1-\epsilon^{1/20})n}{r}$ and $a_2= |W^{k}_i|- n_{\bullet}, b_2= |W^k_{j}|-n_{\bullet}$. This yields a partition $V^k_i \cup U^k_i$ of $W^k_i$ and a partition $V^k_j\cup U^k_j$ of $W^k_j$ satisfying the following.

\begin{equation}\label{eq: regularity 2}
\begin{minipage}{0.9\textwidth} \em
For each $ij \in E(M_k)$, the graphs $G_k[V^k_i, V^k_j], G_k[V^k_i, U^k_j], G_k[U^k_i, V^k_j]$ and $G_k[U^k_i, U^k_j]$ are all $(\epsilon^{1/3},\frac{1}{t})$-super-regular and 
$|V^k_i|=|V^k_j|= n_{\bullet}$.
\end{minipage}
\end{equation}
As $M_k$ is a matching, $V^k_i, U^k_i$ are well-defined for each $i\in V(M_k)$, and we further have 
\begin{align}\label{eq: U sizes}
|U^k_i| \geq (1-3\epsilon)\frac{n}{r} - n_{\bullet} \geq \frac{\epsilon^{1/20}n}{2r}.
\end{align}
Note that the two sets $V^k_i$ and $V^{k'}_i$ (and similarly $U^k_i$ and $U^{k'}_i$) are in general different for $k\neq k'$.
Let $V^k:= \bigcup_{i\in V(M_k)} V^k_i$ and $U^k:= \bigcup_{i\in V(M_k)} U^k_i$. \vspace{0.2cm}

\noindent {\bf Step 2. Applications of  Lemma~\ref{lem: embed one}.}
We arbitrarily partition $\mathcal{T}$ into $\kappa$ collections $\mathcal{T}_{1},\dots, \mathcal{T}_{\kappa}$ such that for all $k\in [\kappa]$, we have 
\begin{align}\label{size1}
e(\mathcal{T}_{k}) < \frac{1}{\kappa}(1- \frac{2\nu}{3})e(G) \le \frac{1}{2rt}(1-\frac{\nu}{2})n^2.
\end{align}
We are ready to construct a desired embedding using Lemma~\ref{lem: embed one}.
For each $k \in [\kappa]$, we will pack 
$\mathcal{T}_k$ into $G_k\cup H[U^k]$. Since $G_1,\dots, G_{\kappa}$ are edge-disjoint, we only have to be careful about disjointness of edges whose images are in $H$.
Suppose that for some $k\in [\kappa]$, we have constructed a function $\phi_{k-1}$ packing $\bigcup_{k'\in [k-1]} \mathcal{T}_{k'}$ into $\bigcup_{k'\in [k-1]} \left(G_{k'} \cup H[U^{k'}]\right)$ satisfying the following, 
\begin{enumerate}
\item[(G1)]\hspace{-0.2cm}$_{k-1}$ $\Delta(H_{k-1})\leq \eta^{-1/70} (k-1)$, where $E(H_{k-1}) :=  \bigcup_{k'\in [k-1]} \phi_{k-1}(\mathcal{T}_{k'})\cap E(H)$.
\end{enumerate}\vspace{0.2cm}
Note that $H_{k-1}$ is the graph consisting of all edges of $H$ which are already used for previous packing.
Observe that (G1)$_{0}$ trivially holds with an empty packing $\phi_0$.
Let $m:= |M_k| \geq (1-t^{-1/3}) \frac{r}{2}$ and  $E(M_k) =:\{ j_{s,1} j_{s,2} : s\in [m]\}$.
For each $(s,i) \in [m]\times [2]$, for brevity, let  
$$V^k_{s,i} := V^k_{j_{s,i} } \enspace \text{and} \enspace U^{k}_{s,i}:= U^k_{j_{s,i}}.$$
Let 
$$G_k^*:= G_{k}[U^k\cup V^k] \cup \left(E(H[U^k]) - E(H_{k-1})\right).$$
We apply Lemma~\ref{lem: embed one} with the following objects and parameters to pack trees in $\cT_{k}$ into $G^*_k$.\newline

\noindent
{ 
\begin{tabular}{c|c|c|c|c|c|c|c|c|c|c|c|c|c|c}
object/parameter& $G_k^*$  & $V_{s,i}^{k}$ & $U_{s,i}^k$ & $ m $ & $n_{\bullet}$ & $\epsilon^{1/6}$  &   $\mathcal{T}_{k}$ & $\Delta$ & $\frac{1}{t}$  & $\frac{\nu}{4}$ & $\eta^{-1/80}$
\\ \hline
playing the role of & $G$   & $V_{s,i}$ & $U_{s,i}$ & $r$ & $n$ & $\epsilon$  & $\mathcal{T}$ & $\Delta$ & $d$  & $\nu$  & $M$ \\ 
\end{tabular}
}\newline \vspace{0.2cm}

For this application, we need to check the conditions of Lemma~\ref{lem: embed one} hold. By \eqref{eq: hierarchy} and \ref{R1},
we have the hierarchy of constants required in Lemma~\ref{lem: embed one} and for each $(s,i)\in [r]\times [2]$, we have $|V^k_{s,i}|= n_{\bullet}$ and $|U^k_{s,i}| \geq \epsilon^{1/10} n_{\bullet}$ by \eqref{eq: U sizes}.
Moreover, each tree $T$ in $\cT_k$ contains at most $(1-\nu)n 
 \leq 2(1-\frac{\nu}{3}) m n_{\bullet}$ vertices by \eqref{eq: matching size} since $\frac{1}{t} \ll \nu$.
 
 Now we show that \ref{lem: embed one 1}-\ref{lem: embed one 3} hold.
Using (G1)$_{k-1}$ and the fact that $k\leq \kappa$, we have
\begin{eqnarray}\label{eq: used up max degree}
\Delta(H_{k-1}) \leq \eta^{-1/70} k \stackrel{\eqref{eq: kappa def}}{\leq} \eta^{-1/60}.
\end{eqnarray}
Then Proposition~\ref{prop: subgraph reg} together with  \eqref{eq: regularity 2} and \eqref{eq: used up max degree} implies that
$G_k^*[U^k_{s,1}, U^k_{s,2}]$, $G_k^*[V^k_{s,1}, U^k_{s,2}]$ and $G_k^*[U^k_{s,1}, V^k_{s,2}]$ are all $(\epsilon^{1/6},\frac{1}{t})$-super-regular.
Since $H_{k-1}$ is edge-disjoint from $G_k$, and $U^k$ is disjoint from $V^k$, we have $G_k^*[V^k_{s,1}, V^k_{s,2}] = G_k[V^k_{s,1}, V^k_{s,2}]$, thus it is $(\epsilon^{1/3},\frac{1}{t})$-super-regular by \ref{eq: regularity 2}.
We conclude that  \ref{lem: embed one 1} holds.
By \eqref{eq: basic assumption}, we have $\widetilde{\alpha}(H[U^k]) \leq \widetilde{\alpha}(H) \leq 2\eta^{1/3}n$. Hence Lemma~\ref{lem: highly connecting sets} implies that for any sets $W,W'\subseteq U^k$ with $|W|,|W'| \geq \eta^{1/10}n$, at least half of vertices $w$ in $W$ satisfies $d_{H[U^k],W'}(w) \geq \eta^{-1/6}$. By  \eqref{eq: used up max degree}, we have $d_{G^*_k,W'}(w)\geq \eta^{-1/6} - \eta^{-1/60}  > 0$.
Hence, $G^*_k$ does not contain any $(\eta^{1/10}n,\eta^{1/10}n)$-bipartite holes and $\widetilde{\alpha}(G^*_k) \leq 2\eta^{1/10}n$, \ref{lem: embed one 2} holds.
Since \eqref{size1} implies that 
 $$e(\cT_{k}) \stackrel{\eqref{size1}}{\leq} \frac{1}{2rt}(1-\frac{\nu}{2})n^2 \leq 
 \frac{1}{2rt} (1-\frac{\nu}{3}) (rn_{\bullet})^2 \leq  \frac{r}{2} (1-\frac{\nu}{3})\frac{1}{t} n_{\bullet}^2
\stackrel{\eqref{eq: matching size}}{ \leq} (1-\frac{\nu}{4})\frac{|M_{k}|}{t} n_{\bullet}^2,$$
 we conclude that \ref{lem: embed one 3} holds. Hence, Lemma~\ref{lem: embed one} gives a map $\phi'$ packing $\mathcal{T}_{k}$ into $G_k^*$ satisfying $d_{\phi'(\mathcal{T}_{k})}(u)\leq \Delta \eta^{-1/80}\leq \eta^{-1/70}$.
Let $\phi_{k}:=\phi_{k-1}\cup \phi'$, then (G1)$_{k}$ holds.
Moreover, by the definition of $H_{k-1}$, $\phi_k$ packs all graphs in $\bigcup_{k'=1}^{k} \mathcal{T}_{k'}$ into $\bigcup_{k'=1}^{k}\left(G_{k'}\cup H[U^{k'}]\right)$.
By repeating this process for each $k=1,\dots, \kappa$, we obtain a function $\phi_{\kappa}$ which packs all trees in $\mathcal{T}$ into $G$. This finishes the proof of Theorem \ref{thm:main}.

\section{Proof of Lemma~\ref{lem: embed one}}
\label{sec: main lemma}

We assume that $G[V] = \bigcup_{s\in [r]} G[V_{s,1}, V_{s,2}]$ as we will not use any other edges in $V$. However, we will use some edges between $U_{s,i}$ and $U_{s',i'}$ with $s\neq s'$ which are guaranteed by \ref{lem: embed one 2}.
We may assume that $\nu <1/3$ and $\Delta\geq 2$.  
By combining two trees of order at most $\frac{2}{3}rn$ with maximum degree at most $\Delta$ into a tree with maximum degree at most $\Delta$ if necessary, we can assume that all trees in $\mathcal{T}$ has at least $\frac{2}{3}rn$ vertices except possibly one. By adding some edges to at most one tree, we may assume that  all trees in $\mathcal{T}$ have at least $\frac{2}{3}rn$ vertices and we have

\begin{align}\label{eq: cT size}
|\cT| \leq \frac{(1-\nu)r dn^2+ 2rn/3 }{ 2rn/3} +1 \leq 3n.
\end{align}

\noindent {\bf Step 1. Preparation of trees.} First, we want to partition each tree $T\in \cT$ into two forests, so that we can later embed each forest into $G$ in different ways.
For each $T\in \cT$, we choose an arbitrary vertex $r_T \in V(T)$ as a root. After applying Proposition~\ref{prop: tree partition} with $T, r_T, n, \Delta$ and $M^{-1/3} n$ playing the roles of $T, r, n, \Delta$ and $t$, respectively, we obtain a collection $\mathcal{S}_T$ of pairwise vertex-disjoint rooted subtrees such that the followings hold.
\begin{enumerate}[label=\text{\rm (S\arabic*)}]
\item \label{S1} $S \subseteq T(x)$ for every $(S, x) \in \mathcal{S}_T.$
\item \label{S2} $M^{-1/3}n \le |S| \le 2\Delta M^{-1/3} n$ for every $(S, x) \in \mathcal{S}_T$. Moreover $ \frac{1}{3\Delta}M^{1/3}r \leq |\mathcal{S}_T| \leq 2r M^{1/3}$.
\item \label{S3} $\bigcup_{(S,x) \in \mathcal{S}_T} V(S) = V(T).$
\end{enumerate} 
Now we will partition each $T\in \mathcal{T}$ into a small forest $C_T$ and a large forest $F_T$ and embed $C_T$ into $G[U]$ and $F_T$ into $G[V]$. For all $T\in \cT$ and $i\in [2]$, we let
\begin{align*}
C_T^0:= \{ x : (S,x)\in \mathcal{S}_T\}, \enspace C_T^{i}:= \bigcup_{x\in C_T^0}D_{S}^{i}(x), \enspace
C_T:= T[C_T^0\cup C_T^1\cup C_T^2] \enspace \text{and} \enspace F_T:= T \setminus V(C_T).
\end{align*}
Since $V(C_T)$ consists of roots of $(S,x) \in \mathcal{S}_T$ and their children and grandchildren in $S$, a vertex $x$ in $C_T^0$ has neighbour $y$ in $V(F_T)$ only when $y$ is a parent of $x$ in $T$ and the
vertices in $C_T^1$ has no neighbours in $V(F_T)$. 
Now we will partition $F_T$ into forests $F_T^1,\dots, F_T^r$ such that each component of $F_T$ lies entirely in one of $F_T^1,\dots, F_T^r$ and for each $s\in [r]$, each $F_T^s$ has a vertex partition $X_1^{T,s} \cup X_2^{T,s}$ satisfying the following.
\begin{enumerate}[label=\text{\rm (S\arabic*)}]
 \setcounter{enumi}{3}
 \item \label{S4} For all $T\in\cT$, $s\in [r]$ and $i\in [2]$, we have $
|X_i^{T,s}| = \frac{1}{2r} |V(T)| \pm \epsilon n \leq (1-3\nu/4) n$.
\item \label{S5} For each $(S,x)\in \mathcal{S}_T$, there exists $(s,i)\in [r]\times [2]$ such that $A_S(x)\setminus V(C_T) \subseteq X_i^{T,s}$ and $B_S(x)\setminus V(C_T) \subseteq X_{3-i}^{T,s}$.
\end{enumerate} 
To see that such a partition exists, we choose $(s,i) \in [r]\times [2]$ independently and uniformly at random for each $(S,x)\in \mathcal{S}_T$, and add $A_S(x)\setminus V(C_T)$ into $X_i^{T,s}$ and $B_S(x)\setminus V(C_T)$ into $X_{3-i}^{T,s}$. Then a simple application of Azuma's inequality shows that \ref{S4} holds with probability at least $0.9$.
Additionally, for each $(S,x)\in \cS_T$, it is clear that \ref{S5} holds. Thus, there exists a vertex partition satisfying both \ref{S4} and \ref{S5}.\vspace{0.2cm}

\noindent {\bf Step 2. Packing small forests $C_T$ into $G[U]$.}
We aim to later embed the vertices in $X_i^{T,s}$ into $V_{s,i}$.
For this, we first embed $C_T$ into $G[U]$ accordingly. Let $(S_1,x_1),\dots, (S_p,x_p)$ be an ordering of $\bigcup_{T\in \cT} \mathcal{S}_T$ such that for each $T\in \cT$, all elements of $\mathcal{S}_T$ appear consecutive in the ordering and  $(S,x)$ comes before $(S',x')$ if $x$ is an ancestor of $x'$ and $(S,x), (S',x') \in \mathcal{S}_T$ for a $T\in \cT$. For each $q \in [p]$, let 
$$C_{q}:= S_q[D^{\leq 2}_{S_{q}}(x_{q})] \enspace\text{and}\enspace W_{q}:= \bigcup_{j \in [q]} V(C_j).$$
We embed trees $C_q$ into $G[U]$ using the following claim.
\begin{claim}
For each $q\in [p]\cup \{0\}$, there exists a function $\phi_{q}$ packing $\{ C_j : j\in [q]\}$ into $G[U]$ satisfying the following for all $T\in \cT$ and $(s,i)\in [r]\times [2]$.
\begin{enumerate}[label=\text{\rm ($\Phi$\arabic*)}]
\item \label{phi1} \hspace{-0.2cm}$_{q}$ For all $y\in X_i^{T,s}$ and $x\in N_{T}(y)\cap W_q$, we have $\phi_q(x)\in U_{s,3-i}$.
\item \label{phi2} \hspace{-0.2cm}$_{q}$ For each $y\in X_i^{T,s}$ with
$N_{T}(y)\cap W_q = \{y_1,\dots, y_b\} \neq \emptyset$, we have 
$d_{G, V_{s,i}}( \phi_q(y_1),\dots, \phi_q(y_b)) \geq (d-\epsilon^{1/2})^{b} n.$

 \item \label{phi3} \hspace{-0.2cm}$_{q}$   For each vertex $u \in U$, we have 
 $|\{ x\in W_{q}: \phi_{q}(x)=u\}| \leq M$.

\end{enumerate}
\end{claim}
\begin{proof}
We use induction on $q$. The statement is trivial if $q=0$. Assume $q\geq 0$ and assume we have $\phi_{q}$ satisfying \ref{phi1}$_{q}$--\ref{phi3}$_{q}$. Let $(S,x):= (S_{q+1},x_{q+1})$. 
Let $T \in \cT$ be the tree containing $S$ and let
$t\in [2 M^{1/3} r]\cup\{0\}$ be the largest number such that
$(S_{q-t+1},x_{q-t+1}),\dots, (S_{q+1},x_{q+1})$ all belong to $\cS_T$.
By \ref{S5}, we let $(s,i) \in [r]\times [2]$ be the index such that 
\begin{align}\label{eq: def of s,i}
B_S(x)\setminus V(C_{q+1})\subseteq X_i^{T,s}.
\end{align}
Let $y:= a_T(x)$, if exists. Note that, by the choice of the ordering $(S_1,x_1),\dots, (S_p,x_p)$, the vertex $y$ (if exists) belongs to one of $S_{q-t+1},\dots, S_{q}$, thus either $y\in W_q$ or $V(F_T)$.

If $y \in W_q$, then let $b:=0$. If $y\in V(F_T)$, then we let
$$\{y_1,\dots, y_b\}:=  N_{T}(y)\cap W_{q}.$$
In other words, $y_i$ is either a child of $y$ which is a root of some $(S_j,x_j)$ with $j\in [q]$ or the parent $a_T(y)$ of $y$ if $a_T(y)$ is in $W_q\cap C_T^2$. Note that $b$ could be zero.
Let $$E:= \bigcup_{j \in [q]}\phi_{q}(E(C_j)) \enspace \text{and} \enspace G':= G-E.$$
In other words, $E$ is the set of all edges in $G$ which have already been used.
Since each tree in $\cT$ has maximum degree at most $\Delta$, ($\Phi$3)$_{q}$ implies that every vertex of $G$ is incident to at most $\Delta M$ edges of $E$.
Note that \eqref{eq: cT size} and \ref{S2} imply that 
$$|E| \leq \sum_{T\in \cT} \sum_{ (S',x')\in \cS_T} |S'| \leq  3n\cdot 2rM^{1/3}\cdot 2\Delta^2= 12 M^{1/3}\Delta^2 r  n.$$
Let 
$$U':= \{ u\in U: |\{ x\in W_q: \phi_{q}(x)=u\}|=M \} \enspace \text{and} \enspace U'':= \bigcup_{j \in [q]\setminus [q-t]} \phi_{j}( V(C_j) ).$$
So $U'$ is a collection of vertices that are ``fully-used'' and $U''$ is a collection of the vertices which is an image of a vertex of the current tree $T$.  In order to obtain ($\Phi$3)$_{q+1}$ as well as to eventually make $\phi_{q+1}$ injective on each forest $F_T$, we want to avoid embedding any vertices in $C_{q+1}$ into $U'\cup U''$.
As every vertex $u\in U'$ is incident to $M$ edges in $E$, so we have $|U'|\le \frac{|E|}{M}$.
As $|C_j| \leq 2\Delta^2$, \ref{S2} implies $|U''| \leq 4\Delta^2 M^{1/3} r$. Hence,
\begin{align}\label{eq: U' U'' size}
|U'\cup U'' | \leq \frac{|E| }{M} +  4\Delta^2 M^{1/3} r\leq  \frac{ 12M^{1/3}\Delta^2 rn}{M }+  4\Delta^2 M^{1/3} r \leq  M^{-1/2} n.
\end{align}
We define $(s^*,i^*)\in [r]\times [2]$ as follows, and we aim to embed $x$ into $U_{s^*,3-i^*}$. This will later ensure ($\Phi$1)$_{q+1}$.
$$ 
(s^*,i^*) := \left\{ \begin{array}{ll} 
(s',i') & \text{ if } y\in V(C_T) \text{ and } \phi_{q}(y)\in U_{s',i'},\\
(s',i') & \text{ if } y \in X^{T,s'}_{i'}, \\
(s',i') & \text{ if } x=r_T \text{ and } D^3_S(x) \subseteq X_{i'}^{s'}.
\end{array}\right.
$$
Note that \ref{S5} ensures that $(s',i')$ exists in the third case when $x=r_T$ and $y$ is not defined.
Recall that the vertex $x$ has at most one neighbour in $V(F_T)$ (its parent $y=a_T(x)$ if belongs to $F_T$) since all of children of $x$ are either non-root vertex in $S$ or a root of some other $(S',x') \in \mathcal{S}_T$. We now define $\phi_{q+1}(x)$ depending on where $y$ lies. We consider the following three cases. \vspace{0.2cm}

\noindent {\bf Case 1.} If $y\in V(C_T)$, then  let $\phi_{q+1}(x)$ be an arbitrary vertex $u$ in $N_{G',U_{s^*,3-i^* }}(\phi_{q}(y)) \setminus (U'\cup U'')$.
By \ref{lem: embed one 1}, we have
\begin{eqnarray*}
|N_{G',U_{s^*,3-i^* }}(y) \setminus (U'\cup U'')|  & \stackrel{\eqref{eq: U' U'' size}}{\geq} & (d-\epsilon)\cdot \epsilon n - \Delta M - M^{-1/2} n \geq 1,
\end{eqnarray*}
hence such a vertex $u$ exists.
Here, we obtain the penultimate inequality from \ref{phi3}. \vspace{0.2cm}

\noindent {\bf Case 2.} If $x=r_T$ and $y$ does not exists, then let $\phi_{q+1}(x)$ be an arbitrary vertex $u$ in $U_{s^*,3-i^* } \setminus (U'\cup U'')$.
Similar argument as Case 1 shows that such a vertex $u$ exists. \vspace{0.2cm}

\noindent {\bf Case 3.} If $y\in V(F_T)$, then let
$$U_*:= U_{s^*,3-i^*} \setminus (U'\cup U'') \enspace \text{ and } \enspace V_{*}:= N_{G,V_{s^*,i^*}}(\phi_{q}(y_1),\dots, \phi_{q}(y_{b})).$$
Recall that $N_{G,V_{s^*,i^*}}(\emptyset) = V_{s^*,i^*}$.
Then we have $|U_*| \geq |U_{s^*,3-i^*}| - M^{-1/2} n\geq |U_{s^*,3-i^*}|/2$. As $y_1,\dots, y_b \in W_q$, the property $(\Phi2)_{q}$ implies
$|V_{*}|\geq(d-\epsilon^{1/2})^b n.$
Since $G[U_{s^*,3-i^*}, V_{s^*,i^*}]$ is $(\epsilon,d)$-regular, Proposition~\ref{prop: reg degree} implies that
at least $|U_*| -\epsilon^{1/2} |U_{s^*,3-i^*}| \geq |U_{s^*,3-i^*}|/3$ vertices $u$ in $U_*$ satisfy $d_{G,V_{*}}(u)\geq (d-\epsilon^{1/2})^{b+1} n$.
We define $\phi_{q+1}(x)$ to be one of such vertices, then ($\Phi$2)$_{q+1}$ holds for the vertex $y\in X_{i^*}^{T,s^*}$.
\vspace{0.2cm}

Now, we want to map vertices in $D_{S}(x)$ to $U_{s^*,i^*}$.
Recall the definition of $s$ and $i$ from \ref{eq: def of s,i}.
Let 
$$W:= N_{G', U_{s^*,i^*}}(\phi_{q+1}(x)) \setminus (U'\cup U'') \enspace \text{and} \enspace W':= U_{s,3-i}\setminus (U'\cup U'').$$
Note that, in any of three cases, \ref{lem: embed one 1} implies $d_{G, V_{s^*,i^*}}(\phi_{q+1}(x))\geq (d-\epsilon^{1/2})n$. Thus \ref{phi3}$_q$ implies that
$$|W|\geq d_{G,U_{s^*,i^*}}(\phi_{q+1}(x)) - \Delta M \geq (d-2\epsilon^{1/2})|U_{s^*,i^*}| \geq \frac{1}{2}\epsilon d n,$$
and we have $|W'|\geq \epsilon n - M^{-1/2}n \geq \frac{1}{2}\epsilon n$.
Lemma~\ref{lem: highly connecting sets} with \ref{lem: embed one 2} implies that at least $\frac{1}{3}\epsilon dn$ vertices $w$ in $W$ satisfy $d_{G,W'}(w)\geq M^2$.
We extend $\phi_{q+1}$ in such a way that $\phi_{q+1}$ maps the vertices in $D_{S}(x)$ into distinct vertices in $W$ each having at least $M^2$ neighbours in $W'$ in the graph $G$.
Since $\phi_{q}$ satisfies \ref{phi3}$_{q}$ and $G'=G-E$, for each $z\in D_{S}(x)$, we have 
$$d_{G',W'}(z) \geq  d_{G,W'}(z) - \Delta M \geq M^2 - \Delta M \geq M.$$
For each $z\in D_{S}(x)$, we define $\phi_{q+1}$ on $D_{S}(z)$ in such a ways that $\phi_{q+1}$ maps the vertices in $D_{S}(z)$ into different vertices in $N_{G',W'}(z)$ and $\phi_{q+1}$ is still injective on vertices in $S$.
This is possible as $|C_{q+1}|\leq \Delta^2 < M\leq d_{G',W'}(z)$.
By our construction, $\phi_{q+1}$ embeds $C_{q+1}$ into $G[U]-E$, thus $\phi_{q+1}$ packs $C_{1},\dots, C_{q+1}$ into $G[U]$.

Now we check that $\phi_{q+1}$ satisfies \ref{phi1}$_{q+1}$--\ref{phi3}$_{q+1}$. Note that any vertex in $V(F_T) \cap N_{T}^1(C_{q+1})$ is either $y$ or vertices in $D_{S}^3(x)$. If $y\in V(F_T)$, our choice of $\phi_{q+1}$ and the definition of $(s^*,i^*)$ ensure that \ref{phi1}$_{q+1}$ holds for the vertex $y$. As $D^{3}_{S}(x)\subseteq B_S(x)\setminus D^{\leq 2}_S \subseteq X_i^{T,s}$ and
$\phi_{q+1}(D_{S}^2(x)) \subseteq W' \subseteq U_{s,3-i}$, \ref{phi1}$_{q+1}$ holds for the vertices in $D_{S}^3(x)$. For all vertices outside $N_{T}^1(C_{q+1})$, \ref{phi1}$_{q}$ implies \ref{phi1}$_{q+1}$. Hence \ref{phi1}$_{q+1}$ holds.
Note that the definition of $\phi_{q+1}(x)$ in Case 3 ensures that \ref{phi2}$_{q+1}$ holds for $y$, if $y \in V(F_T)$. For vertices in $D_{S}^{3}(x)$, by the definition of $W'$ and \ref{lem: embed one 1}, \ref{phi2}$_{q+1}$ holds for the vertices in $D_{S}^3(x)$ with $b=1$. Again, for all vertices outside $N_{T}^1(C_{q+1})$, \ref{phi2}$_{q}$ implies \ref{phi2}$_{q+1}$, hence \ref{phi2}$_{q+1}$ holds.
Since we have not mapped any vertices into $U'$ and every vertex in $D_{S}^{\leq 2}(x)$ is injectively mapped, \ref{phi3}$_{q+1}$ holds. This finishes the induction and the proof of the claim.
\end{proof}

\noindent {\bf Step 3. Packing forests $F^s_T \subseteq F_T$ into $G[V_{s,1},V_{s,2}]$.} Let $\phi:=\phi_p$ given by the above claim. From now on, \ref{phi1}--\ref{phi3} denote \ref{phi1}$_{p}$--\ref{phi3}$_{p}$.
We wish to pack each forest $F_T$ into $G[V]$ by using Theorem~\ref{Blowup} in such a way that the vertices in $X^{T,s}_i$ embed into $V_{s,i}$, then all edges of $F_T$ lie in $\bigcup_{s\in [r]} G[V_{s,1}, V_{s,2}]$.
To use Theorem~\ref{Blowup}, we first  pack trees into internally regular bipartite graphs. 
Since $F_T^1,\dots, F_T^r$ are vertex-disjoint subforests of tree $T$, the property \ref{phi1} ensures that we can consider $F_T^s$ for each $s \in [r]$ separately to pack into $G[V_{s,1},V_{s,2}]$. 

Moreover, we want the obtained packing of $F_T$ to be consistent with $\phi$, so the neighbours of already embedded vertex $x$ of $T$ are also embedded to a neighbour of $\phi(x)$ in $G$. We will define sets $W_i^F$ and $Y_i^{s'}$ for this purpose, and we will use \ref{phi2} together with \ref{Blowup A3} to obtain this consistency. 

We choose a new integer $q$ and a constant $\zeta$ satisfying $\epsilon \ll 1/q \ll \zeta \ll d,\nu.$ 
We fix  a number $s\in [r]$ throughout Step 3, and let 
$$\cF^s:= \{ F_T^s : T\in \cT\}.$$
Because a forest contains less edges than vertices,
\ref{S4} and \ref{lem: embed one 3} imply that 
\begin{align}\label{eq: number of edges in cF}
e(\cF^s) \leq \sum_{T\in \cT} \left(\frac{1}{r} |V(T)| + 2\epsilon n \right) \leq (1-\frac{2\nu}{3}) dn^2.
\end{align}
For all $F =F_T^s \in \cF^s$ and $i\in [2]$, let $X^F_i:= X^{T,s}_i$, then $X^F_1 \cup X^F_2$ is a vertex partition of $F$ into two independent sets.
By \ref{S4}, for all $F\in \cF^s$ and $i\in [r]$, we have  $|X^{F}_i | \leq (1-3\nu/4) n$.
Let 
\begin{eqnarray}\label{w}
w := \frac{e(\mathcal{F}^s) }{(1-4\zeta)qn} \stackrel{ \eqref{eq: number of edges in cF}}{\le} \frac{(1-2\nu/3)dn^2}{(1 - 4\zeta)qn} \le \frac{d}{q} (1 - \frac{\nu}{2})n.\end{eqnarray}
We partition $\mathcal{F}^s$ into collections $\mathcal {F}_1, \dots, \mathcal {F}_w$ so that we have the following for each $s' \in [w]$.
\begin{eqnarray}\label{edgesum}
e(\cF_{s'})=\sum_{F\in \mathcal{F}_{s'}} e(F) = \frac{e(F^s)}{w} \pm 4n \leq (1 - 3\zeta)qn.
\end{eqnarray}
Note that this is possible because $e(F)\leq 2n$ for each $F\in \cF^s$.
Since each $T\in \cT$ has at least $\frac{2}{3}rn$ vertices, \ref{S4} implies that  for each $s' \in [w]$
\begin{eqnarray}\label{q} 
|\mathcal{F}_{s'}| \le \frac{e(\cF_{s'})}{\frac{1}{2r}\cdot (2rn/3)-\epsilon n}\le \frac{qn}{n/4}=4q \le (q \zeta)^{3/2}. 
\end{eqnarray} 
For each $i\in [2]$ and $F=F_T^s\in \mathcal{F}$, we define
$$W_{i}^F:= N^1_{T}( V(C_T) ) \cap X^F_i.$$
In other words, $W_i^{F}$ is the collection of the vertices which have neighbours already embedded by $\phi$, so we need a special care  when we embed the vertices in $W_i^{F}$ to make sure we embed edges of $T$ incident to the vertices in $W_i^{F}$ into edges of $G$. For all $i\in [2]$ and $F \in \mathcal{F}^s$, we have
\begin{eqnarray}\label{eq: Wi+1T size}
|W_{i}^F| \leq \Delta |V(C_T)| \leq 2\Delta^3 |\mathcal{S}_T| \stackrel{\ref{S2}}{\leq } M.
\end{eqnarray}
Note that we have $M\leq \epsilon n$.
Thus, by \eqref{edgesum}, \eqref{q} and \eqref{eq: Wi+1T size}, for each $s'\in [w]$,  we can apply Lemma~\ref{Pack} with $\cF_{s'}, X^F_{i}, W_{i}^{F}, q, \zeta, \epsilon$ and $|\cF_{s'}|$ playing the roles of $\cL, X_i^j, W^j_i, k, \zeta,\epsilon$ and $s$, respectively.
Then for each $s'\in [w]$, we obtain a function $\Phi_{s'}$ packing forests in $\cF_{s'}$ into a $2n$-vertex $q$-regular graph $H_{s'}$ with a balanced bipartition $X_1 \cup X_2$.
Moreover, for all $i\in [2]$ and $F,F'\in \cF_{s'}$, 
\begin{eqnarray}\label{eq: W disjoint}
\Phi_{s'}(W_i^{F}) \cap \Phi_{s'}(W_{i}^{F'}) = \emptyset.
\end{eqnarray} 
For all $i \in [2]$ and $s' \in [w]$, we let
$$Y_i^{s'}:= \bigcup_{F \in \mathcal{F}_{s'} } \Phi_{s'}(W_i^{F}) \text{ and } Y^{s'}:=\bigcup_{i\in [2]} Y_i^{s'}.$$
 By \eqref{q} and \eqref{eq: Wi+1T size},  for each $s' \in [w]$ and $i\in [2]$, we have
 \begin{eqnarray} \label{eq: Yij size}
 |Y_i^{s'}| \le 4q\cdot M \leq \epsilon n. 
 \end{eqnarray}
We now wish to use Theorem \ref{Blowup}  to pack $\mathcal{H}^s := \{H_1, \dots, H_w\}$ into $G[V_{s,1},V_{s,2}]$. This packing combined with $\Phi_{s'}$ would give us a packing of $F_T$ into $G$. Moreover, we want the edges of $T$ between $V(C_T)$ and $V(F_T)$ to be edge-disjointly mapped into $E(G)$. Note that the vertices in $Y_i^{s'}\subseteq X_i$ are the images of vertices that is incident to such edges between $V(C_T)$ and $V(F_T)$.
For each $y\in Y_i^{s'}$, let $x_y$ be the preimage of $y$, i.e.~$\Phi_{s'}(x_y)=y$, and let $T_y \in \mathcal{T}$ be the tree containing $x_y$. Let $N_y:= N_{T_y}(x_y)\cap V(C_{T_y})$ and 
$$ A^{s'}_y:= N_{G, V_{s,i}}( \phi(N_y) ).$$
Since a vertex $y$ is an image of $x_y$, mapping $y$ to a vertex $v$ means that $x_y$ will be embedded into $v$ in our final packing. Since $N_y$ is the set of already embedded neighbours of $x_y$, the vertex $v\in V_{s,i}$ must be a  the common neighbour (in $G$) of the vertices in $\phi(N_y)$. Therefore, $A^{s'}_y$ is the set of vertices which we can embed $y$ into. 
By \ref{phi2}, we have
\begin{align}\label{eq: target set}
|A^{s'}_y| \geq  (d-\epsilon^{1/2})^{\Delta} n.
\end{align}
In addition, there is one more issue to consider.
If there are two vertices $x\in C_T$ and $x'\in C_{T'}$ from different trees $T\neq T'$ satisfies $\phi(x)=\phi(x')$ and we have two vertices $y\in N_{T}(x)\cap F_T$ and $y'\in N_{T'}(x')\cap F_{T'}$, we cannot embed $y$ and $y'$ into the same vertex. Note that, by (\ref{eq: W disjoint}), we do not need to worry about conflicts between two vertices from different trees in the same collection $\cF_{s'}$. 
To deal with this overlapping issue for trees from different collections, we consider the following auxiliary graph  $\Gamma$ with 
$V(\Gamma) := \{(s',y): s'\in [w] \text{ and } y\in  Y^{s'} \}$ and 
$$E(\Gamma) : = \{(s',y)(s'',y')  : s'\neq s''\in [w], 
y \in Y^{s'}, y' \in Y^{s''} \text{ and } \phi(N_y) \cap \phi(N_{y'}) \neq\emptyset\}.$$

For a fixed pair $(s',y)\in V(\Gamma)$ and fixed $s''\in [w]\setminus \{s'\}$, we have
\begin{align}\label{eq: q2 upperbound}
&\hspace{-0.2cm} |\{ y' \in Y^{s''}: (s'',y') \in N_{\Gamma}( (s',y) ) \}| \leq |\{ y' : F_{T_{y'}}^s \in \cF_{s''},  \phi(N_y)\cap \phi(N_{y'})\neq \emptyset \}| \nonumber \\
& \leq   \sum_{ u\in \phi(N_y) } |\{y': F_{T_{y'}}^s\in \cF_{s''},  u\in \phi(N_{y'})  \}|
 \leq \sum_{ u\in \phi(N_y)} \sum_{F_{T}^{s} \in \cF_{s''} } |\{ x \in V(F_{T}^{s}) :
u\in \phi(N_{T}(x)\cap C_T)   \}|  \nonumber \\
& \leq   \sum_{ u\in  \phi(N_y) } \sum_{F_T^s \in \cF_{s''} } |\{ x \in N_T(x') : u= \phi(x') \}| 
\leq \sum_{ v\in  \phi(N_y)} \sum_{F_T^s \in \cF_{s''}} \Delta
\leq \Delta^2|\cF_{s''}|  
\stackrel{\eqref{q} }{\leq} \Delta^2 (\zeta q)^{3/2} \leq q^2.
\end{align}
Moreover, for any $(s',y)\in V(\Gamma)$,  we have
 \begin{align}\label{eq: Delta(F)}
d_{\Gamma}((s',y)) \leq \sum_{u\in \phi(N_y)}|\phi^{-1}(u)| \stackrel{\ref{phi3}}{\leq} \Delta M \leq \epsilon n. 
 \end{align}
Now we apply Theorem~\ref{Blowup} with the following objects and parameters. \newline

\noindent
{ 
\begin{tabular}{c|c|c|c|c|c|c|c|c|c|c|c|c|c|c}
object/parameter& $G[V_{s,1},V_{s,2}]$ & $V_{s,i}$ & $H_{s'}$&$w$ &$q$& $\epsilon$&$n$ & $d$& $(d-\epsilon^{1/2})^{\Delta}$ & $\frac{\nu}{10}$ & $\Gamma$ & $Y_i^{s'}$ & $ A_y^{s'}$ \\ \hline
playing the role of &$ G$ & $V_i$ & $H_j$ &$s$ &$k$& $\epsilon$ &$n$ & $d$ & $d_0$ & $\alpha$ & $\Gamma$ & $Y_i^j$ & $A_{y}^{j}$
\end{tabular}
}\newline \vspace{0.2cm}

Indeed, \ref{lem: embed one 1} implies that \ref{Blowup A1},  and \ref{Blowup A2} holds by the definition of $H_s$.
Properties \eqref{eq: Yij size} and \eqref{eq: target set} imply that \ref{Blowup A3} holds, and 
 the properties \eqref{eq: Yij size}, \eqref{eq: q2 upperbound}  and \eqref{eq: Delta(F)} imply that \ref{Blowup A4} holds with the above parameters.
Thus by Theorem \ref{Blowup} we obtain a function $\phi^*_{s}$ which packs $\{H_1,\dots, H_w\}$ into $G[V_{s,1},V_{s,2}]$  satisfying the following for each $s'\in [w]$ and $i\in [r]$.\begin{enumerate}[label=\text{\rm ($\Phi^*$\arabic*)}]
\item\label{Phi*1}$\phi^*_s(y) \in A_y^{s'}$ for all $y \in Y_1^{s'}\cup Y_2^{s'}$, 
\item\label{Phi*2} For all $(s',y)(s'',y') \in E(F),$ we have  $\phi^*_s(y) \ne \phi^*_s(y').$
\end{enumerate}
So $\phi^s:= \phi^*_s(\bigcup_{s'\in [w]}\Phi_{s'})$ packs $\cF^{s}$ into $G[V_{s,1}, V_{s,2}]$.\vspace{0.2cm}

\noindent {\bf Step 4. Combining the functions.} 
By repeating Step 3 for all $s\in [r]$, we obtain functions $\phi^1,\dots, \phi^{r}$ packing all forests in $\cF^{1},\dots, \cF^{r}$. Let $\phi' :=\phi\cup \bigcup_{i\in [r]} \phi^i$.
Then $\phi'$ packs every forest in $\cF^s$ into $G[V_{s,1},V_{s,2}]$, thus into $G[V]$. Since  $\phi'$ is also an extension of $\phi$,  for all $T\in \cT$, $\phi'$ packs $C_T$  into $G[U]$.
Moreover, \ref{Phi*1}, \ref{Phi*2} and the definitions of $A^{s'}_y$ and $\Gamma$ imply that $\phi'$ packs edges in
$\{e \in E(T[V(C_T), V(F_T)]) : T\in \mathcal{T}\}$
into distinct edges in $\bigcup_{(s,i)\in [r]\times [2]} G[U_{s,i}, V_{s,3-i}] $. Thus we conclude that $\phi'$ packs $\mathcal{T}$ into $G$. Moreover, \ref{phi3} implies that for each $u\in U$, we have $d_{\phi(\cT)}(u) \le\Delta\cdot M$. This finishes the proof of Lemma~\ref{lem: embed one}.

\bibliographystyle{amsplain}
\bibliography{littrees}

\medskip

{\footnotesize \obeylines \parindent=0pt
	
	\begin{tabular}{lllll}
		Jaehoon Kim	&\ & Younjin Kim	&\ &	Hong Liu        \\
		School of Mathematics &\ & Institute of Mathematical Sciences &\ & Mathematics Institute		  		 	 \\
		University of Birmingham &\ & Ewha Womans University &\ & University of Warwick 	  			 	 \\
		Birmingham &\ & Seoul &\ & Coventry                             			 \\
		UK &\ & South Korea &\ & UK						      \\
	\end{tabular}
}

\begin{flushleft}
	{\it{E-mail addresses}:
		\tt{j.kim.3@bham.ac.uk, younjinkim@ewha.ac.kr, h.liu.9@warwick.ac.uk.}}
\end{flushleft}

\end{document}